\documentclass[12pt]{article}
\usepackage[]{preprint-style} 
\usepackage{preprint-notation}
\usepackage{empheq}

\newcommand*\widefbox[1]{\fbox{\hspace{1em}#1\hspace{1em}}}
\title{Low Regularity Estimates for CutFEM Approximations 
of an Elliptic Problem with Mixed Boundary Conditions
\funding{The 
Swedish Foundation for Strategic Research Grant No.\ AM13-0029, the Swedish Research Council Grants Nos.\ 2013-4708, 2017-03911 and the Swedish Research Programme Essence}}

\date{\today}
\author{
Erik Burman\footnote{\addressucl}
\mbox{ } 
Peter Hansbo\footnote{\addressju}
\mbox{ } 
Mats G. Larson\footnote{\addressumu} 
}
\date{\today}

\newcommand{\supp}{\text{supp}}

\newcommand{\bd}{{\partial \Omega_D}}
\newcommand{\bn}{{\partial \Omega_N}}
\begin{document}

\maketitle

\begin{abstract}
We show error estimates for a cut finite element approximation of a second order elliptic 
problem with mixed boundary conditions. The error estimates are of low regularity type 
where we consider the case when the exact solution $u \in H^s$ with $s\in (1,3/2]$. For 
Nitsche type methods this case requires special handling of the terms involving the normal 
flux of the exact solution at the the boundary. For Dirichlet boundary
conditions the estimates are optimal, whereas in the case of mixed
Dirichlet-Neumann boundary conditions they are suboptimal by a
logarithmic factor.
\end{abstract}

\section{Introduction}
In this paper we will consider the finite element approximation of the Poisson problem with 
mixed boundary conditions under minimal regularity assumptions. Let $\Omega$ be a 
domain in $\IR^d$ with smooth boundary $\partial \Omega$, which is decomposed into two subdomains $\bd$ and $\bn$ such that
$\partial \Omega 
= \overline{\partial \Omega}_D \cup \bn 
= \bd \cup \overline{\partial \Omega}_N$ and $\bd \cap \bn = \emptyset$.
Consider the problem: find $u: \Omega \rightarrow \IR$ such that 
\begin{alignat}{3}\label{eq:strong-mixed-a}
-\Delta u &=f & \qquad &\text{in $\Omega$}
\\ \label{eq:strong-mixed-b}
u &=g_D & \qquad &\text{on  $\bd$}
\\ \label{eq:strong-mixed-c}
\nabla_n u &=g_N & \qquad &\text{on  $\bn$}
\end{alignat}
where $f:\Omega \rightarrow \IR$, $g_D:\Gamma_D \rightarrow \IR$ and
$g_N:\Gamma_N \rightarrow \IR$ satisfy the following bound for $s>1$,
\begin{equation}\label{eq:data-reg}
\|f \|_{H^{s-2}(\Omega)} 
 + 
 \|g_D \|_{H^{s-1/2}(\partial\Omega_D)}
 + 
 \|g_N\|_{H^{s-3/2}(\bn)} \lesssim 1
\end{equation}
%\todo{need a bit more on $f$ close to the boundary}
Here and below we used the notation $a \lesssim b$ for $a \leq C b$, with $C$ a
positive constant. 

For the approximation of the problem we apply a Cut Finite Element Method (CutFEM).  
In CutFEM the boundary is allowed to cut through the computational cells in an (almost) arbitrary way and stabilization terms are added in the vicinity of the boundary to ensure 
that the method is coercive and that the resulting linear system of equations is invertible. 

In previous work on fictitious domain finite element methods see \cite{BH12,BCHLM15}, error estimate were shown under the assumption that $u \in H^s(\Omega)$ with $s>3/2$. The objective of the present work is to relax this regularity requirement. Indeed, we show an a 
priori error estimates in the energy norm, requiring only that $u \in H^s(\Omega)$, where $s>1$, and $\Delta u $ is in $L^2(U_{\delta_0})$ on some arbitrarily thin neighborhood $U_{\delta_0}$ of the Dirichlet boundary $\partial \Omega_D$. Since the test functions in the Nitsche formulation of the Dirichlet condition are 
not zero on $\partial \Omega_D$, we will also have to choose the Neumann data 
$g_N$ in a slightly smaller space than $H^{-1/2}(\partial \Omega_N)$. \textcolor{black}{We focus our attention on the effects of rough data in CutFEM.  We assume that the  boundary $\partial \Omega$ of the domain $\Omega$ is smooth and 
that we can evaluate integrals on the intersection of simplices and the domain 
and its boundary, exactly.} Estimation of the error resulting from approximation of the domain can be handled using the techniques in \cite{BHLZ16}.

The study of the convergence of nonconforming methods for the
approximation of solution with low regularity has received increasing
interest since the seminal paper by Gudi \cite{Gudi10}. In that work
optimal convergence for low regularity solutions were obtained using 
ideas from a posteriori error analysis, where the error is upper
bounded by certain residuals of the discrete solution. These residuals are 
then shown to lead to optimal upper bounds using the discrete local efficiency bounds. A
similar approach was used by L\"uthen et al. \cite{LJS17} for a generalised
Nitsche's method on fitted meshes. This approach does however not seem to be
suitable for the case of cut finite element method since for cut elements the 
local efficiency bounds are not robust with respect to the mesh boundary intersection. 
Instead, in the spirit of \cite{EG18}, we use a
version of duality pairing to handle the term involving the normal
flux of the interpolation error. This is made more delicate by the
presence of mixed boundary conditions. Indeed to include this case in
the analysis we introduce a regularized bilinear form and use the solution to the
regularized problem as pivot in the error estimate. The regularization
gives rise to a logarithmic factor. Observe that this is due to the mixed boundary
conditions. For pure Dirichlet conditions or pure Neumann conditions
the analysis results in optimal error bounds for $s\ge1$.

The paper is organized as follows: In Section 2 we introduce the functional framework 
for the model problem and formulate the finite element method and in Section 3 we 
derive the error estimates.

\section{Weak Formulation and the Finite Element Method}
Since we consider low regularity solutions of a problem with mixed
boundary condition we must be careful with the fractional Sobolev
spaces for the traces of the functions. In this section we first
introduce the notations and definitions for the functional analytical
framework, leading to the weak formulation of
\eqref{eq:strong-mixed-a}--\eqref{eq:strong-mixed-c}. Then we introduce
the cut finite element method for the approximation of the weak solutions.

\subsection{Function Spaces}
Let $\omega \subset \IR^n$ and let $H^s(\omega)$ denote the usual Sobolev 
spaces on $\omega$. Define 
\begin{align}
H^{1/2}(\partial \Omega) &= H^1(\Omega)|_{\partial \Omega}
%&=\{ v: \partial \Omega \rightarrow \IR | v = w|_{\partial \Omega}, w \in H^1(\Omega) \}
\\
\| v \|_{H^{1/2}(\partial \Omega )} 
&= \inf_{w \in H^1(\Omega), w|_{\partial \Omega} = v} \| w \|_{H^1(\Omega)}
\end{align}
and for $\Gamma \subset \partial \Omega$, define
\begin{align}
H^{1/2}(\Gamma) &= H^{1/2}(\partial \Omega)|_\Gamma
%\\
%\widetilde{H}^{1/2}(\Gamma) &= \{ v \in H^{1/2}(\Gamma) : \supp(v) \subset \Gamma\}
\\
\| v \|_{H^{1/2}(\Gamma)}
&=
% \| v \|_{\widetilde{H}^{1/2}(\Gamma)}
 \inf_{w \in H^1(\Omega), w|_\Gamma = v} \| w \|_{H^1(\Omega)}
\end{align}
and the subspace 
\begin{align}
%H^{1/2}(\Gamma) &= H^{1/2}(\partial \Omega)|_\Gamma
%\\
\widetilde{H}^{1/2}(\Gamma) &= \{ v \in H^{1/2}(\Gamma) : \supp(v) \subset \Gamma\} 
\subset H^{1/2}(\Gamma)
%\\
%\| v \|_{H^{1/2}(\Gamma)}
%&= \| v \|_{\widetilde{H}^{1/2}(\Gamma)}
%= \inf_{w \in H^1(\Omega), w|_\Gamma = v} \| w \|_{H^1(\Omega)}
\end{align}
Then $H^{1/2}(\partial \Omega\setminus \Gamma)
= H^{1/2}(\partial \Omega)/ \widetilde{H}^{1/2}(\Gamma)$ 
and $H^{1/2}(\partial \Omega) = H^{1/2}(\partial \Omega\setminus \Gamma) 
\oplus \widetilde{H}^{1/2}(\Gamma)$. Next define the dual spaces
\begin{align}
H^{-1/2}(\partial \Omega) &= [ H^{1/2}(\partial \Omega) ]^*
\\
H^{-1/2}(\Gamma) &= [ \widetilde{H}^{1/2}(\Gamma) ]^*
\\
\widetilde{H}^{-1/2}(\Gamma) &= [ {H}^{1/2}(\Gamma) ]^*
\end{align}
consisting of functionals $g:X \rightarrow \IR$ with duality pairing 
$\langle g,v \rangle_{X^*\times X} = g(v)$ and norm 
\begin{equation}
\| g \|_{X^*} = \sup_{ v \in X \setminus \{ 0 \} } \frac{g(v)}{\| v \|_X}
\end{equation}
with $X\in \{ H^{1/2}(\partial \Omega), H^{1/2}(\Gamma),  \widetilde{H}^{1/2}(\Gamma) \}$.  We 
will use the simplified 
notation $(g,v)_\Gamma$ for the duality pairing. Note that for each $g \in H^{-1/2}(\Gamma)$ 
we may define $\widetilde{g}\in H^{-1/2}(\partial \Omega)$ by $\widetilde{g}(v) = g(v|_\Gamma)$ 
and thus $H^{-1/2}(\Gamma) \hookrightarrow H^{-1/2}(\partial \Omega)$.
 
\subsection{Weak Formulation}
The problem (\ref{eq:strong-mixed-a})--(\ref{eq:strong-mixed-c}) can be
cast on weak form: find $u \in V_{g_D}$ such that 
\begin{equation}\label{eq:weak-problem}
a(u,v) = l(v) \qquad v \in V_0
\end{equation}
where 
\begin{equation}\label{eq:weak-forms}
a(v,w) = (\nabla v,\nabla w)_\Omega, \qquad l(v) 
= (f,v)_{\partial \Omega} + (g_N,v)_{\bn} 
\end{equation}
and, for each $ g_D  \in H^{1/2}(\bd)$, 
\begin{equation}
V_{g_D} = \{ v \in H^1(\Omega) : v|_{\bd} = g_D \}
\end{equation}
For $f \in H^{-1}(\Omega)$, $g_D \in H^{1/2}(\partial \Omega_D)$, and 
$g_N \in H^{-1/2}(\partial \Omega_N)$, there
exists a unique weak solution to (\ref{eq:weak-problem}) and the following 
elliptic regularity estimate holds, $1 \leq s < 3/2$,
 \begin{equation}\label{eq:regularity}
 \| u \|_{H^{s}(\Omega)} \lesssim 
 \|f \|_{H^{s-2}(\Omega)} 
 + 
 \|g_D \|_{H^{s-1/2}(\partial\Omega_D)}
 + 
 \|g_N\|_{H^{s-3/2}(\bn)}
 \end{equation}
We refer to Savar\'e \cite{Sav97} for a precise characterization of the
regularity for the mixed problem.

\subsection{The Normal Flux}

The normal flux $\nabla_n u  = n \cdot \nabla u \in H^{-1/2}(\partial \Omega)$, where $n$ is 
the exterior unit normal, plays an important role in what follows. For
\textcolor{black}{$u \in H^1(\Omega)$, with $\Delta u \in L^2(\Omega)$}, it
can be defined by the identity 
\begin{equation}\label{eq:flux_def}
(\nabla_n u, v )_{\partial \Omega}  = (\Delta u, v )_\Omega + (\nabla u, \nabla v)_\Omega 
\qquad \forall v \in H^1(\Omega)
\end{equation}

Observe that in the finite element method we work with weakly enforced boundary conditions 
and therefore we will not have test functions that vanish on $\bd$, i.e. the test functions are not in $V_0$, and herefore we will consider boundary data such that 
\begin{equation}
g_D \in H^{1/2}(\bd), \qquad g_N \in \widetilde{H}^{-1/2}(\partial \Omega_N)
\end{equation}
where the Neumann data $g_N$ is chosen in the smaller space 
$\widetilde{H}^{-1/2}(\partial \Omega_N) \subset {H}^{-1/2}(\partial \Omega_N)$, 
compared to the strong formulation and corresponding weak form (\ref{eq:weak-problem}). 
We will also assume that the source term $f$ is square integrable over some (arbitrary thin) neighbourhood of the boundary $\partial \Omega$, see (\ref{eq:defV}) below.
%Using the Sobolev embedding theorem we note that 
%\begin{equation}
%H^{-1/2}(\bn) \subset L^{(d-1)/{2d}}(\bn), 
%\qquad 
%H^{1/2}(\bd) \subset L^{(d+1)/{2d}}(\bd)
%\end{equation}
%\todo[inline]{Fix text a bit better in this section, add details on the boundedness, and references}

\subsection{Finite Element Method}
To define the cut finite element method let $\Omega_0$ be a polygonal domain such 
that $\Omega \subset \Omega_0$ and let $\{\mcT_{h,0}: h \in (0,h_0]\}$ be a family of 
quasiuniform meshes covering $\Omega_0$ with mesh 
parameter $h := \max_{T \in \mcT_{h,0}} \mbox{diam}(T)$. \textcolor{black}{For a subset $\omega\subset \Omega_0$, define the
submesh of elements intersecting $\omega$, by $\mcT_h(\omega):=\{ T \in
\mcT_{h,0} : T \cap \omega  \neq \emptyset\}$, and let $\mcT_h := \mcT_h(\Omega)$} be the 
so called active mesh. Let $V_{h,0}$ be the conforming finite element space defined on
  $\mcT_{h,0}$ consisting of piecewise
  affine functions and define $V_h = V_{h,0}|_{\mcT_h}$. 
Define the bilinear forms 
\begin{align}\label{eq:Ah}
A_h(v,w) &:= a(v,w) 
- (\nabla_n v,w)_{\bd} 
- (v, \nabla_n w)_{\bd} 
+\beta h^{-1} (v,w)_{\bd}
\\ \label{eq:sh}
s_h(v,w) &:= \sigma h([ \nabla_n v],[\nabla_n w])_{\mcF_h(\partial \Omega)}
\\ \label{eq:Lh}
L_h(v) &:= (f, v)_\Omega  
+ (g_N,v)_{\bn}
- (g_D, \nabla_n v)_{\partial\Omega_D} 
+ \beta h^{-1} (g_D,v)_{\bd}
\end{align}
with positive parameters $\beta$ and $\sigma$, 
$\mcF_h(\partial \Omega)$ the set of interior faces in $\mcT_h$ 
associated with an element $T \in \mcT_h(\partial \Omega) 
= \{ T \in \mcT_h : T \cap \partial \Omega \neq \emptyset\}$ that 
intersects the boundary, and the jump in the normal flux 
at face $F$ shared by elements $T_1$ and $T_2$ is defined by 
\begin{equation}
[\nabla_n v ] = \nabla_{n_1} v_1 + \nabla_{n_2} v_2  \qquad \text{on $F$}
\end{equation}\label{eq:jump}
where $v_i = v|_{T_i}$ and $n_i$ is the unit exterior normal.

Define the finite element method: find  $u_{h} \in V_{h}$ 
 such that 
\begin{equation}\label{eq:method}
A_{h}(u_{h},v) + s_h(u_h,v) = L_h(v)\qquad \forall v \in V_{h}
\end{equation}

%, it is characterized in the following
%Lemma.
%\begin{lem}(Galerkin orthogonality)
%Let $u$ detnote the solution of \eqref{eq:weak} and $u_h$ the solution
%of \eqref{eq:method}. Then
%\begin{equation}\label{eq:galort}
%A_h(u - u_h, v ) = s_h(u_h,v) \qquad \forall v \in V_h
%\end{equation}
%\end{lem}
%\begin{proof}
%First observe that using the definition \eqref{eq:flux_def}
%\begin{align}
%a(u,v) 
%& - (\nabla_n u,v)_{\bd} 
%- (u, \nabla_n v)_{\bd} 
%+\beta h^{-1} (u,v)_{\bd} \\
%& = (-\Delta u, v)_\Omega + (\nabla_n g_N,v)_{\bn} 
%                                          -(g_D, \nabla_n v)_{\partial
%                                          \Omega_D} +\beta h^{-1}
%                                          (g_D,v)_{\bd}
%  \\
%& =  (f, v)_\Omega + (\nabla_n g_N,v)_{\bn} 
%                                          -(g_D, \nabla_n v)_{\partial
%                                          \Omega_D} +\beta h^{-1}
%                                          (g_D,v)_{\bd}
%  = L_h(v).
%\end{align}
%Using this relation and the definition of the method \eqref{eq:method}
%we conclude that 
%\[
%A_h(u - u_h, v ) = L_h(v) -A_{h}(u_{h},v) = s_h(u_h,v).
%\]
%\end{proof}
%

\section{Error Analysis}
In this section we will derive the error estimates, here as usual the
consistency of the method is of essence. However, for solutions with
low regularity this is delicate in the case of mixed boundary conditions. 
Indeed, in the low regularity case, \eqref{eq:flux_def} is not sufficient to make
sense of the term $(\nabla_n u,w)_{\bd}$ \textcolor{black}{for
  approximation purposes, since the division on $\bd$
and $\bn$ necessarily results in a boundary integral over one of the
subdomains that has to be lifted in some other fashion. This is
problematic since the solution is not regular enough to allow for the usual trace inequality arguments}. To handle this difficulty we introduce a regularized finite element 
formulation (for analysis purposes only), where a smooth weight function $\chi$ is
introduced and the problematic term is replaced by 
\begin{equation}
(\nabla_n u,w)_{\chi,\partial \Omega} := (\chi \nabla_n u,w)_{\partial \Omega}
\end{equation}
The regularized method has a consistency error that can be controlled
by sharpening the cut off function $\chi$.

\subsection{Outline}
We shall prove low regularity energy norm error estimates using the 
following approach:
\begin{itemize}
\item Similarly to \cite{Gudi10} we estimate the error in a norm which 
does not involve the $L^2$ norm of the normal trace of the gradient.
\item For the case of mixed boundary conditions, we introduce a regularized bilinear form and the corresponding
  (nonconsistent) finite element 
method. The regularization takes the form of a weight function
smoothing the transition from the Dirichlet to the Neumann boundary
condition in the first boundary integral of the form $A_h$, see
equation \eqref{eq:Ah}.  In the regularized norm we can use a version of $H^{-1/2}-H^{1/2}$ 
duality in an $\epsilon$ neighborhood of $\bd$.
\item The total error is estimated using a Strang type argument. The
  error is divided into the approximation error, the discrete error
  between an interpolant and the finite element solution of the
  regularized formulation and finally the regularization error between
  the regularized and standard finite element solutions.
\end{itemize}

\subsection{The Cut Off Function}
Key to the regularized problem is the design of the weight
function, \textcolor{black}{ $\chi:\Omega \rightarrow \mathbb{R}$ with support in a
neighbourhood of $\bd$}. This function takes the value $1$ on $\bd$ and
decays smoothly to zero in an $\epsilon$ neighbourhood of
$\overline{\partial \Omega}_D \cap \overline{\partial \Omega}_N$ 
and into the domain away from the boundary. This
way it plays the role of a cut off, that localizes the boundary
integral to $\bd$, while the form remains well defined
for low regularity solutions. In order to define the cut off function 
we introduce some notation. 

\paragraph{Notation.} \textcolor{black}{For $x \in \IR^d$, $\omega \subset \IR^d$, 
let $\rho_\omega(x)>0$ be the distance function $\rho_\omega(x) = \text{dist}(x,\omega)$ 
and let $p_\omega:\IR^d\rightarrow \omega$ be the closest point mapping.
In the case $\omega \equiv \partial \Omega$ we drop the subscript.} For 
$\delta \in (0,\delta_0]$, define the $\delta$-neighbourhood of $\partial \Omega$, 
\begin{equation}\label{eq:Udelta-bnd}
U_\delta(\partial \Omega) 
= \{ x \in \Omega : \rho(x) < \delta\}
\end{equation}
Then there is $\delta_0>0$ such that the closest point mapping 
$p:U_{\delta_0}(\partial \Omega) \rightarrow \partial \Omega$ maps every $x$ to 
precisely one point at $\partial \Omega$. We also define $\delta$-neighbourhood 
of $\bd$ and $\bn$ as follows
\begin{equation}\label{eq:Udelta-D}
U_\delta(\bd) = \{ x \in U_\delta(\partial \Omega) : p(x) \in \bd \}, 
\qquad 
U_\delta(\bn) =U_\delta \setminus U_\delta(\bd) 
\end{equation}
Let $\Sigma  = \partial (\bd) = \partial (\bn)$ be the smooth interface separating 
$\bd$ and $\bn$ and let $\nu$ be the unit conormal to $\Sigma$ exterior to 
$\bn$ and tangent to $\partial \Omega$. See Figure \ref{fig:a}. For $t \in [0,\delta_0]$ let
\begin{align}
\partial \Omega_{t} &= \{ x \in \Omega : \rho(x) = t \}
\\
\partial \Omega_{N,t} &= \{ x \in \partial \Omega_t : p(x) \in \bn \}
\\
\Sigma_t &= \{ x \in \partial \Omega_t : p(x) \in \Sigma \}
\end{align}
Note that $p: \partial \Omega_t \rightarrow \partial \Omega$ is a bijection 
for all $t \in [0,\delta_0]$.
Let 
\begin{equation}\label{eq:Utgamsigt}
U_{t,\gamma}(\Sigma_t) = \{ x \in \partial \Omega_{N,t} :\rho_{\Sigma_t}(x) 
< \gamma \} \subset \partial \Omega_{N,t}
\end{equation}
be the $\gamma$ tubular neighborhood of $\Sigma_t$ in $\partial \Omega_{N,t}$, 
and assume that $\gamma \in (0,\gamma_0]$ with $\gamma_0$ small enough to guarantee that the closest point mappings $p_{\Sigma_t}$ are well defined for all $t\in [0,\delta_0]$, and let 
\begin{equation}\label{eq:Ugamsig}
U_{\gamma}(\Sigma)= U_{0,\gamma}(\Sigma_0) \subset \partial \Omega_N 
\end{equation}
Define 
\begin{equation}\label{eq:Udeltaepsilon}
U_{\delta,\epsilon} = \cup_{t \in [0,\delta]} U_{t, \gamma(t)}(\Sigma_t)
\end{equation}
with $\gamma(t)=t+\epsilon$ for $\epsilon \in (0,\epsilon_0]$ 
and $\epsilon << \delta$, see Figure \ref{fig:b}. \textcolor{black}{Defining, for $z\in \Sigma$, }
\begin{equation}\label{eq:udeltaeps-z}
\textcolor{black}{U_{\delta, \epsilon}(z) = \{ x \in U_{\delta, \epsilon} : p_\Sigma (x) = z \}}
\end{equation}
\textcolor{black}{where $p_\Sigma$ is the closest point mapping associated with $\Sigma$, 
we have $U_{\delta,\epsilon} = \cup_{z \in \Sigma } U_{\delta,\epsilon}(z)$}. Note that 
$U_{\delta, \epsilon}(z) = U_{\delta,\epsilon} \cap p_\Sigma^{-1}(z) \subset U_{\delta_0}(\Sigma)\cap p_\Sigma^{-1}(z) $, which is a subset of the $2$ dimensional normal space $N_\Sigma(z)$ to the $d-2$ dimensional tangent space $T_\Sigma(z)$ of $\Sigma$ at $z$. In the case $d=2$, $\Sigma$ consists of distinct points and in that case $U_{\delta,\epsilon} \subset U_{\delta_0}(z) \subset p_\Sigma^{-1} (z)$, for $\delta_0$ small enough. Finally, let
\begin{equation}\label{eq:Udelta}
U_\delta = U_\delta (\bd) \cup U_{\delta,\epsilon}
\end{equation}

\begin{figure}
\centering
\includegraphics[scale=1.0]{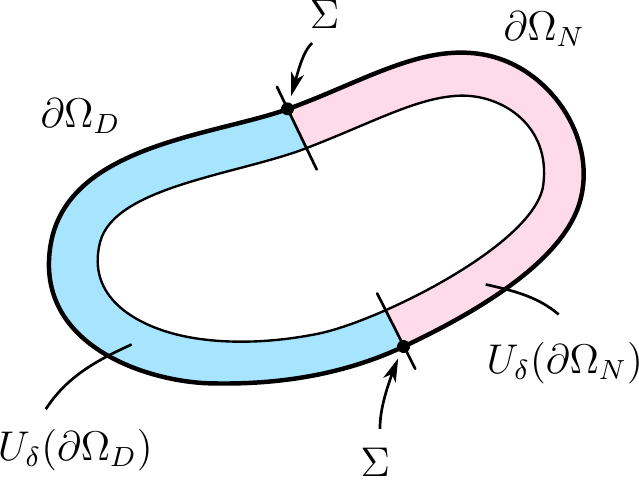}
\hspace{1cm}
\includegraphics[scale=1.0]{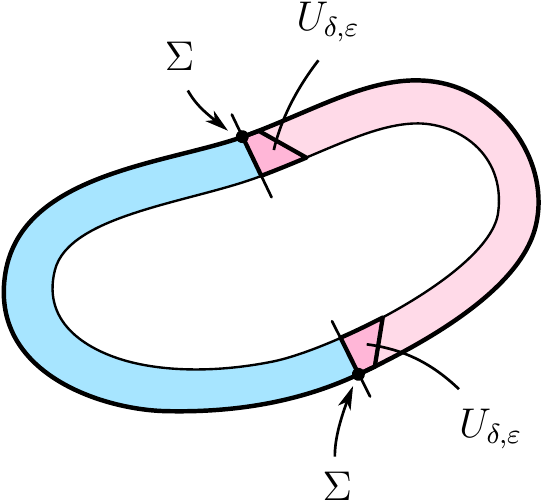}
\caption{Left: the Dirichlet boundary $\partial \Omega_D$, the Neumann boundary $\partial \Omega_N$, the interface $\Sigma$, and the tubular neighborhood $U_\delta(\partial \Omega) = U_\delta(\partial \Omega_D) \cup U_\delta(\partial \Omega_N)$. Right: the set $U_{\delta,\epsilon}\subset U_\delta(\partial \Omega_N)$.}
\label{fig:a}
\end{figure}
\begin{figure}
\includegraphics[scale=1.0]{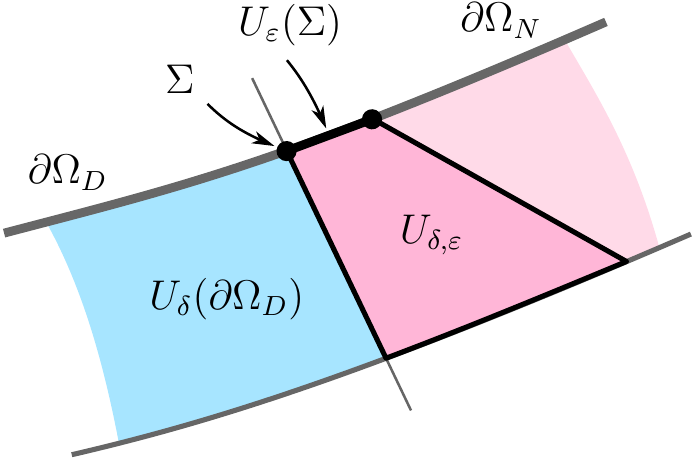}
\hspace{1cm}
\includegraphics[scale=1.0]{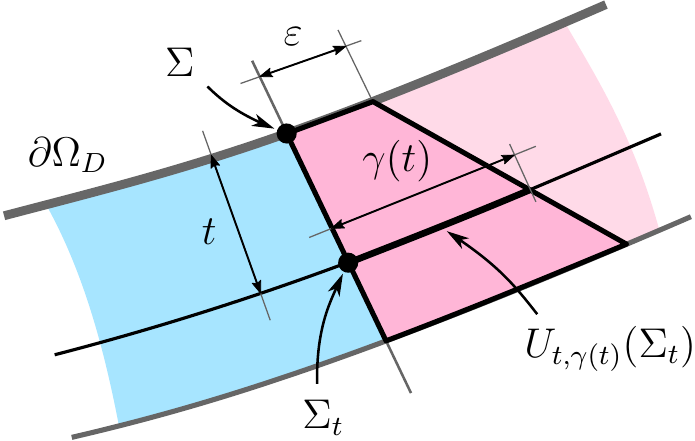}
\caption{Left: Close up of the set $U_{\delta,\epsilon}$ including $U_\epsilon(\Sigma) 
\subset \partial \Omega_N$. Right: The set $\Sigma_t$ and $U_{t,\gamma}(t)(\Sigma_t)$.} 
\label{fig:b}
\end{figure}

\paragraph{The Cut Off Function.}
We will below take $\delta \sim h$ and $\epsilon \sim h^\alpha$ with $\alpha = d$. Let 
$\chi:\Omega \rightarrow [0,1]$ be smooth such that
\begin{equation}\label{eq:cutoff}
\begin{cases}
\chi = 1 \quad \text{on $\bd$}
\\
\chi = 0 \quad \text{on $\bn  \setminus U_{\epsilon}(\Sigma)$}
\\ 
\chi = 0 \quad \text{on $\Omega \setminus U_\delta$}
\end{cases}
\qquad 
\begin{cases}
\|\nabla \chi \|_{L^\infty(U_\delta \setminus U_{\delta,\epsilon})} \lesssim \delta^{-1}
\\
\|\nabla_n \chi \|_{L^\infty(U_{\delta,\epsilon})} \lesssim \delta^{-1}
\\
\| \nabla_{\Sigma} \chi \|_{L^\infty(U_{\delta,\epsilon})} \lesssim 1
\\
\| \nabla_{\nu} \chi \|_{L^\infty(U_{t,\gamma(t)})} \lesssim (\gamma(t))^{-1}\quad t \in [0,\delta]
\end{cases}
\end{equation}

Observe that in the definition above $\nabla_\Sigma$ denotes the
projection of the gradient on the tangent plane of $\Sigma$. By the
construction of $\chi$, $\| \nabla_{\Sigma} \chi \|_{L^\infty(U_{\delta,\epsilon})}$ 
is bounded and depends only on $\epsilon$, $\delta$ and the regularity of $\Sigma$. 
%The function $\chi$ can be constructed by first changing coordinates locally to a flat 
%boundary and then integrating a properly scaled and translated bump function of the form 
%$c \exp(-1/(1-(|x|/\epsilon)^2)

\begin{lem}
The cut off function $\chi$ satisfies the following estimate 
\begin{equation}\label{eq:cutoff-bound}
\sup_{z\in \Sigma} \| \nabla_\nu \chi \|^2_{ U_{\delta,\epsilon}(z)}  
\lesssim 
|\ln (1+\delta/\epsilon)|
\end{equation}
and with   
\begin{equation}\label{eq:eps-delta}
\delta \sim h, \qquad \epsilon \sim h^\alpha 
\end{equation}
for $ 1 \leq  \alpha \lesssim 1$ we obtain
\begin{equation}\label{eq:cutoff-bound-b}
\| \nabla_\nu \chi \|^2_{ U_{\delta,\epsilon}}  
\lesssim 
1+ |\ln(h)|
\end{equation}
\end{lem}
\begin{proof} Using the bounds for $\nabla_\nu \chi$ we obtain
\begin{align} \nonumber
\|\nabla_\nu \chi \|^2_{U_{\delta,\epsilon}(z)} 
&=\int_{U_{\delta,\epsilon}(z)} | \nabla_\nu \chi |^2
 \lesssim \int_0^\delta \int_{U_{t,t+\epsilon}(z)} (t+\epsilon)^{-2} 
\\ \nonumber
&\qquad  \lesssim \int_0^\delta (t+\epsilon)^{-1} 
=[\ln(t+\epsilon)]^\delta_0
=\ln (1+\delta/\epsilon)
\end{align}
Estimate (\ref{eq:cutoff-bound-b}) follows directly from the definition 
of $\delta$ and $\epsilon$. 
\end{proof}

\subsection{The Regularized Problem}
For $\epsilon \in (0,\epsilon_0]$ define the regularized form 
\begin{align}\label{eq:Ah-eps}
A_{h,\epsilon}(v,w) 
&= (\nabla v, \nabla w)_\Omega
- (\nabla_n v,w)_{\chi,\partial \Omega} 
- (v, \nabla_n w)_{\partial \Omega_D} 
+\beta h^{-1} (v,w)_{\partial \Omega}
\end{align}
and define $A_{h,0} = A_h$. We will show that the mapping 
$[0,\epsilon_0]  \ni \epsilon \mapsto A_{h,\epsilon}$ is continuous for $\epsilon$ small enough, 
see Lemma \ref{lem:form-error} below for details.

For $\epsilon \in [0,\epsilon_0]$ define the regularized finite element method: 
find  $u_{h,\epsilon} \in V_{h}$ such that 
\begin{equation}\label{eq:method-eps}
A_{h,\epsilon}(u_{h,\epsilon},v)  + s_h(u_h,v) = L_h(v)\qquad \forall v \in V_{h}
\end{equation} 
This method is not consistent, but we have the identity
\begin{align}
A_{h,\epsilon}(u - u_{h,\epsilon}, v) 
&= A_{h,\epsilon}(u, v) - L_h(v) + s_h(u_h,v)
\\ \label{eq:residual-regularized}
&= s_h(u_h,v) -(g_N,v)_{\chi,\bn} \qquad \forall v \in V_h
\end{align}
since using Green's formula gives
\begin{align}
&A_{h,\epsilon}(u,v) = (\nabla u, \nabla v)_\Omega
- (\nabla_n u,v)_{\chi,\partial \Omega} 
- (u, \nabla_n v)_{\partial \Omega_D} 
+\beta h^{-1} (u,v)_{\partial \Omega}
\\
&= -(\Delta u, v)_\Omega + (\nabla_n u , v)_{\partial \Omega}
- (\nabla_n u,v)_{\chi,\partial \Omega} 
- (u, \nabla_n v)_{\partial \Omega_D} 
+\beta h^{-1} (u,v)_{\partial \Omega}
\\
&=(f,v)_\Omega - (g_D,\nabla_n v)_{\partial \Omega_D} 
+ \beta h^{-1}  (g_D,v)_{\partial \Omega_D}  + (g_N,v)_{\partial \Omega_N}
- (g_N,v)_{\chi,\partial \Omega_N}
\\
&=L_h(v) - (g_N,v)_{\chi,\partial \Omega_N}
\end{align}
where we used the fact that $\chi = 1$ on $\partial \Omega_D$ to conclude that 
\begin{align}
(\nabla_n u , v)_{\partial \Omega}
- (\nabla_n u,v)_{\chi,\partial \Omega} 
&= (\nabla_n u , v)_{\partial \Omega_N}
- (\nabla_n u,v)_{\chi,\partial \Omega_N}
\\
&=(g_N,v)_{\partial \Omega_N} - (g_N, v)_{\chi,\partial \Omega_N}
\end{align}

\subsection{Properties of the Bilinear Forms}

We here summarize the basic results on the bilinear forms and 
conclude with a proof of existence, uniqueness, and stability of the 
finite element solutions.  

\paragraph{Inverse and Trace Inequalities.} Let us recall some inverse and trace
inequalities. Here $\mathbb{P}_1(T)$ denotes the set of polynomials of
degree less than or equal to $1$ on the simplex $T$.
\begin{itemize}
\item Inverse 
inequalities (see \cite[Section 1.4.3]{DiPE12}),
\begin{equation}\label{inverse}
\|\nabla v\|_{H^1(T)} \lesssim h^{-1}_T \|v\|_{L^2(T)}\quad \forall v \in \mathbb{P}_1(T)
\end{equation}
and
\begin{equation}\label{Linfty}
\|v\|_{L^{\infty}(T)} \lesssim h^{-\frac{d}{2}} \|v\|_{L^2(T)}\quad \forall v \in \mathbb{P}_1(T)
\end{equation}
\item Trace inequalities (see \cite[Section 1.4.3]{DiPE12}),
\begin{equation}\label{trace_H1}
\|v\|_{L^2(\partial T)} \leq C_T \left( h_T^{-{1/2}} \|v\|_{L^2(T)} +
h_T^{{1/2}} \| \nabla v \|_{T}\right)\quad \forall v \in H^1(T)
\end{equation}
and 
\begin{equation}\label{trace_P1}
\|v\|_{L^2(\partial T)} \leq C_t h_T^{-{1/2}} \|v\|_{L^2(T)} \quad\forall v \in \mathbb{P}_1(T)
\end{equation}
\item Inverse trace inequality on cut elements. For a simplex $T$ such that $T \cap \partial \Omega \ne \emptyset$, there holds
%\begin{equation}\label{trace_H1_cut}
%\|v\|_{L^2(T \cap \partial \Omega)} \lesssim \left( h_T^{-{1/2}} \|v\|_{L^2(T)} +
%h_T^{{1/2}} \|\nabla v\|_{T}\right)\quad \forall v \in H^1(T)
%\end{equation}
%and 
\begin{equation}\label{trace_P1_cut}
\|v\|_{L^2(T \cap \partial \Omega)} \lesssim h_T^{-{1/2}} \|v\|_{L^2(T)} \quad\forall v \in \mathbb{P}_1(T)
\end{equation}
%
%\item Trace inequalities on cut elements, see \cite{HanHanLar03}. For a simplex $T$ such that $T \cap \partial \Omega \ne \emptyset$, there holds
%\begin{equation}\label{trace_H1_cut}
%\|v\|_{L^2(T \cap \partial \Omega)} \lesssim \left( h_T^{-{1/2}} \|v\|_{L^2(T)} +
%h_T^{{1/2}} \|\nabla v\|_{T}\right)\quad \forall v \in H^1(T)
%\end{equation}
%and 
%\begin{equation}\label{trace_P1_cut}
%\|v\|_{L^2(T \cap \partial \Omega)} \lesssim h_T^{-{1/2}} \|v\|_{L^2(T)} \quad\forall v \in \mathbb{P}_1(T)
%\end{equation}
\end{itemize}

\paragraph{Stabilization Estimates.}
For any two elements $T_1$ and $T_2$ in $\mcT_h$, sharing a face $F$, we have the 
estimate   
\begin{equation}\label{eq:inverse-jumps-pair}
\| \nabla^m v \|^2_{T_1} \lesssim \| \nabla^m v \|^2_{T_2} + h^{3-2m} \| [\nabla_n v ] \|^2_F 
\qquad m=0,1, \quad v \in V_h
\end{equation}
Repeated use 
of (\ref{eq:inverse-jumps-pair}) leads to
\begin{equation}\label{eq:inverse-jumps}
\| \nabla v \|^2_{\mcT_h} \lesssim \| \nabla v \|^2_\Omega + \| v \|^2_{s_h} \qquad v \in V_h
\end{equation}
For sets $\omega_0 \subset \omega_1 \subset \Omega$ such that 
$\mbox{diam}(\omega_1 \setminus \omega_0) \lesssim h$, we may also 
derive the estimate 
\begin{equation}\label{eq:L2inverse_jumps}
\|\nabla^m  v \|^2_{\mcT_h(\omega_1)} 
\lesssim \| \nabla^m v \|^2_{\mcT_h(\omega_0)} 
+ h^{3-2m} \| [\nabla_n v ] \|^2_{\mcF_h(\omega_1)}  
\qquad m=0,1, \quad v \in V_h
\end{equation}
where $\mcF_h(\omega_1)$ denotes the interior faces of $\mcT_h(\omega_1)$. 

\paragraph{The Energy Norm.} We equip the finite element space $V_h$ with the 
energy norm
\begin{align}\label{eq:energy-norm}
\tn v \tn^2_{h} 
&= \| \nabla v \|^2_\Omega 
+ \| v \|^2_{s_h} + h^{-1} \| v \|^2_{\bd}
\end{align}
where $\| v \|^2_{s_h}:=s_h(v,v)$. In order to have the normal flux well defined on the Dirichlet boundary we assume that 
\begin{equation}\label{eq:defV}
v \in V = \{ v \in H^1(\Omega) : \Delta v |_{U_{\delta_0}} \in L^2(U_{\delta_0})\}
\end{equation}
where we recall, see (\ref{eq:Udelta}), that $\text{supp}(\chi) 
\subset U_{\delta_0} = U_{\delta_0} (\partial \Omega_D ) \cup U_{\delta_0,\epsilon_0}$ 
for all regularization parameters $\epsilon \in [0,\epsilon_0]$. The stabilization form $s_h$ 
is not defined on $V$, due to the low regularity, and therefore we equip  $V$ 
with the weaker energy norm 
\begin{equation}\label{eq:energy-norm-nostab}
\tn v \tn^2
= \| \nabla v \|^2_\Omega 
+ h^{-1} \| v \|^2_{\bd}
\end{equation}

\begin{lem} There is constant such that for all $v  \in V + V_h$, $w \in V_h$, and 
$\epsilon \in [0,\epsilon_0]$, 
\begin{empheq}[box=\widefbox]{equation}\label{eq:continuity}
A_{h,\epsilon}(v,w) 
\lesssim \tn v \tn \, \tn w \tn_{h} 
+ |(\nabla_n v, w)_{\chi,\partial \Omega} |
\end{empheq}
where we use the norm $\tn \cdot \tn$, which does not include the stabilization, 
on $V+V_h$. 
\end{lem}
\begin{proof}
To verify this estimate we start from the definition (\ref{eq:Ah-eps}) 
of the regularized form and using the Cauchy Schwarz inequality we get 
%\todo{sh}
\begin{align}
A_{h,\epsilon}(v,w) 
&\lesssim \|\nabla v\|_\Omega  \|\nabla w\|_\Omega 
+ |(\nabla_n v,w)_{\chi,\partial \Omega}| 
\\ 
&\qquad 
+ h^{-1/2} \|v\|_{\partial \Omega_D} h^{1/2} \| \nabla_n w\|_{\partial \Omega_D} 
+\beta h^{-1} \|v\|_{\partial \Omega} \|w\|_{\partial \Omega}
\\
&\lesssim \tn v \tn \, \tn w \tn_h + |(\nabla_n v,w)_{\chi,\partial \Omega}| 
\end{align}
We estimated $h^{1/2} \| \nabla_n w\|_{\partial \Omega_D}$, with $w\in V_h$,  using 
the inverse inequality 
\begin{equation}\label{eq:inverse-normal}
h \| \nabla_n w \|^2_{\bd} 
\lesssim \| \nabla v \|^2_{\mcT_h(\bd)}
\lesssim \| \nabla w \|^2_{\mcT_h} 
\lesssim \| \nabla w \|^2_\Omega + \| w \|^2_{s_h}
\lesssim \tn w \tn^2_h
\end{equation}
where we first used the inverse trace inequality (\ref{trace_P1_cut}) and 
then the stabilization estimate (\ref{eq:inverse-jumps}).
\end{proof}

We will now prove a bound on the error introduced by replacing $A_h$
by its regularized counterpart $A_{h,\epsilon}$.
\begin{lem}\label{lem:form-error}
There is a constant such that for all $v,w\in V_h$, and $\epsilon \in [0,\epsilon_0]$ 
with $\epsilon_0 \sim h$, 
\begin{empheq}[box=\widefbox]{equation}\label{eq:form-error}
|A_{h,\epsilon}(v,w) - A_h(v,w )| \lesssim \epsilon h^{1-d} \tn v \tn_{h} \tn w \tn_{h}
\end{empheq}
\end{lem}
\begin{proof} Using the definitions (\ref{eq:Ah}) and (\ref{eq:Ah-eps}) of the forms 
$A_h$ and $A_{h,\epsilon}$ we obtain
\begin{align}
|A_h(v,w) - A_{h,\epsilon}(v,w)|
&= 
|(\nabla_n v, \chi w )_{\bn}|
\\ 
&\lesssim
h^{1/2}\|\nabla_n v\|_{U_\epsilon(\Sigma)} 
h^{-1/2}\| w \|_{U_\epsilon(\Sigma)}
\\ 
&\lesssim 
\epsilon h^{1-d} \tn v \tn_{h} \tn w \tn_{h}
\end{align}
where  we used the fact that $\text{supp}(\chi)\cap \partial \Omega_N \subset U_\epsilon(\Sigma)$, see (\ref{eq:Ugamsig}). To estimate $h \| \nabla_n v \|^2_{U_\epsilon(\Sigma)}$ we proceed in the same way as in (\ref{eq:inverse-normal}), 
we first use an inverse estimate and then the stablization (\ref{eq:inverse-jumps}),
\begin{align}
h \| \nabla_n v \|^2_{U_\epsilon(\Sigma)}  
\lesssim 
h \| \nabla_n v \|^2_{\mcT_h(U_\epsilon(\Sigma))\cap \bn}  
\lesssim
\| \nabla v \|^2_{\mcT_h(U_\epsilon(\Sigma))}
\lesssim
\| \nabla v \|^2_{\mcT_h}
\lesssim \tn v \tn^2_{1,h}
\end{align}
Next to estimate $h^{-1} \| v \|^2_{U_\epsilon(\Sigma)}$ we pass over to the $L^\infty$ 
norm in order to extract an $\epsilon$ factor and then we use suitable inverse bounds to 
pass to the energy norm.
\begin{align} 
h^{-1} \| v \|^2_{U_\epsilon(\Sigma)} 
&\lesssim h^{-1} \epsilon \| v \|^2_{L^\infty(U_\epsilon(\Sigma))} 
\\ 
&\lesssim h^{-1} \epsilon \| v \|^2_{L^\infty(\mcT_h(U_\epsilon(\Sigma)))} 
\\ 
&\lesssim h^{-1} \epsilon h^{-d} \| v \|^2_{\mcT_h(U_\epsilon(\Sigma))}
\\ \label{eq:eqv-bb}
&\lesssim h^{-1} \epsilon h^{-d} \Big( h\| v \|^2_{\bd 
\cap {\widetilde{\mcT}_h(U_\epsilon(\Sigma))}
} 
+ h^2 \| \nabla v \|^2_{\widetilde{\mcT}_h(U_\epsilon(\Sigma))}\Big)
\\ 
&\lesssim  \epsilon h^{1-d} \Big( h^{-1}\| v \|^2_{\bd} 
+ \| \nabla v \|^2_{{\mcT}_h} \Big)
\\ \label{eq:eqv-cc}
&\lesssim \epsilon h^{1-d} \tn v \tn^2_{h} 
\end{align}
Here $\widetilde{\mcT}_h(U_\epsilon(\Sigma))$ is a slightly larger patch of elements 
such that the $d-1$ dimensional measure of its intersection with the Dirichlet boundary 
satisfies $|\widetilde{\mcT}_h (U_\epsilon(\Sigma)) \cap \bd |  \sim h^{d-1}$ ,which allows 
us to utilize the control available in $\tn v \tn_h$ at the Dirichlet boundary and to employ 
a Poincar\'e inequality in (\ref{eq:eqv-bb}), see the appendix in \cite{BHL18}. The patch $\mcT_h(U_\epsilon(\Sigma))$ does not in general satisfy ${\mcT}_h (U_\epsilon(\Sigma)) \cap \bd  \sim h^{d-1}$ and therefore it is enlarged by adding a suitable number 
of face neighboring elements in $\mcT_h(\partial \Omega_D)$.  In the last 
step (\ref{eq:eqv-cc}) we also used the stabilization (\ref{eq:inverse-jumps}). 
Note that due to the assumption 
that $\epsilon \in [0,\epsilon_0]$ with $\epsilon_0 \sim h$ it follows from shape regularity 
that there is a uniform bound on the number of elements in $\widetilde{\mcT}_h (U_\epsilon(\Sigma))$.
\end{proof}

Lemma \ref{lem:form-error} is instrumental for the coercivity that we prove next.
\begin{lem}
For $\beta$ large enough and $\sigma>0$, the forms 
$
A_{h,\epsilon}+s_h$, $h \in (h,h_0]$, $\epsilon \in [0, c h^{d}]$ 
with $c$ small enough, are coercive 
\begin{empheq}[box = \widefbox]{equation}\label{eq:coercivity}
\tn v \tn^2_{h} \lesssim A_{h,\epsilon}(v,v) +s_h(v,v) \qquad v \in V_h
\end{empheq}
\end{lem}
\begin{proof} First we note that $A_{h,0}$ is coercive using standard techniques together 
with the inverse estimate (\ref{eq:inverse-normal}). Next using the
bound (\ref{eq:form-error}) of Lemma \ref{lem:form-error}, 
we obtain 
\begin{align}\nonumber
A_{h,\epsilon}(v,v) 
&= A_{h,0}(v,v) + A_{h,\epsilon}(v,v) - A_{h,0}(v,v)
\\ \nonumber
&\geq C_1 \tn v \tn_h^2 - |A_{h,\epsilon}(v,v) - A_{h,0}(v,v)|
\\ \nonumber
&\geq (C_1 - C_2 \epsilon h^{1-d}) \tn v \tn_h^2
\\ \nonumber
&\gtrsim \tn v \tn_h^2
\end{align}
where in the last step we choose $\epsilon \leq c h^{d}$ with $h \in (0,h_0]$ and 
$c$ small enough.
\end{proof}
%\todo[inline]{Choice of $\epsilon$? Looks like we could take $\epsilon = c h^{d-1}$ 
%with $c$ small enough. Of course the particular choice is not important.}

Using Lax-Milgram we conclude that for each $\epsilon \in [0,c h^d]$, there 
is a unique solution $u_{h,\epsilon} \in V_h$ to the regularized problem 
(\ref{eq:method-eps}) such that 
\begin{equation}\label{eq:discrete-stab-bound}
\tn u_{h,\epsilon} \tn_h \lesssim \sup_{v \in V_h \setminus \{0\} } L_h(v) 
\lesssim 
\| f \|_{H^{-1}(\Omega)} 
+ \|g_N\|_{\widetilde{H}^{-1/2}(\bn)} 
+ h^{-1/2} \| g_D \|_{\bd}
\end{equation}

\subsection{Technical Lemmas}
In this section we collect some technical results that will be useful
in the analysis. More precisely we start with four technical lemmas 
before proving Lemma \ref{lem:trace-cont} which is used to estimate 
the problematic term $(\nabla_n v, w)_{\chi, \partial \Omega}$ in the 
regularized problem.

\begin{lem} \label{lem:sobolev_sigma} There is a constant such that for all $v \in V_h$,
\begin{equation}\label{eq:dscrete-sobolev}
\int_\Sigma \| v \|^2_{L^\infty(U_{\delta_0,\epsilon_0}(z))} \lesssim (1 + |\ln(h)|) \, \tn v \tn^2_{h}
\end{equation}
%\todo{form ln factor?}
\end{lem}
\begin{proof} {\bf 1.} Recall that for $z \in \Sigma$, $U_{\delta, \epsilon}(z) = \{ x \in U_{\delta, \epsilon} : p_\Sigma (x) = z \}$, see (\ref{eq:udeltaeps-z}), and we have $U_{\delta,\epsilon} = \cup_{z \in \Sigma } U_{\delta,\epsilon}(z)$. There are $\delta_0 \sim \epsilon_0\sim 1$ such that 
$\delta \in (0,\delta_0]$, $\epsilon \in (0,\epsilon_0]$ and 
\begin{equation}\label{eq:udeltaeps-inclusion}
U_{\delta,\epsilon}(z) \subset U_{\delta_0,\epsilon_0}(z)
\end{equation}
We shall first show that there is a constant such that for all $z \in \Sigma$, 
\begin{equation}\label{eq:dscrete-sobolev-a}
\| v \|^2_{L^\infty(U_{\delta_0,\epsilon_0}(z))}
\lesssim
(1 + |\ln(h)|) \| v \|^2_{H^1(U_{\delta_0,\epsilon_0}(z))} 
+ h^2 \| \nabla v \|^2_{L^\infty(U_{\delta_0,\epsilon_0}(z))}
\end{equation}
To that end note that $U_{\delta_0,\epsilon_0}$ has the following cone property: for each 
$x \in U_{\delta_0,\epsilon_0}(z)$ there is a cone (or sector since $U_{\delta_0,\epsilon_0}$ 
is two dimensional) $\Lambda_{r_0}(x) \subset U_{\delta_0,\epsilon_0}(z)$, with vertex $x$, radius $r_0 \sim \delta_0 \sim 1$, and opening angle $\theta_0 \sim 1$. For 
$x \in U_{\delta_0,\epsilon_0}(z)$ 
and $r,\theta \in \Lambda_{r_0}(x)$ we have  the identity 
\begin{equation}\label{eq:discrete-sobolev-b}
v(x) = v(r,\theta) - \int_0^r \partial_r v(s,\theta) ds
\end{equation}
and the estimate
\begin{equation}\label{eq:discrete-sobolev-cc}
v^2(x) 
\lesssim v^2(r,\theta) + \left( \int_0^{r_0} \partial_r v(s,\theta) ds \right)^2
\end{equation}
We estimate the integral on the right hand side as follows
\begin{align}
\left(\int_0^{r_0} \partial_r v(s,\theta) ds\right)^2 
&\lesssim
\left( \int_0^{\eta h}  \partial_r v(s,\theta) ds \right)^2  
+ \left( \int_{\eta h}^{r_0} \partial_r v(s,\theta) ds \right)^2
\\ \label{eq:discrete-sobolev-ccc}
& \lesssim 
(\eta h)^2 \| \nabla v \|^2_{L^\infty(\Lambda_{\eta h})}  
+|\ln(d/\eta h)|  \int_{\eta h}^{r_0} (\partial_r v(s,\theta))^2  s ds
\end{align}
where for the second term on the right hand side we used the estimate 
\begin{align}
 \left( \int_{\eta h}^{r_0} \partial_r v(s,\theta) ds \right)^2
& \lesssim
  \int_{\eta h}^{r_0} s^{-1} ds  \int_{\eta h}^{r_0}  (\partial_r v(s,\theta))^2 s ds
\\
&
  \lesssim
  |\ln(d/\eta h)| \int_{\eta h}^{r_0} (\partial_r v(s,\theta))^2  s ds
\end{align}
Combining (\ref{eq:discrete-sobolev-cc}) and (\ref{eq:discrete-sobolev-ccc}), we 
get
\begin{align}
v^2(x)  \lesssim v^2(r,\theta) +
(\eta h)^2 \| \nabla v \|^2_{L^\infty(\Lambda_{\eta h})}  
+|\ln(r_0/\eta h)|  \int_{\eta h}^{r_0} (\partial_r v(s,\theta))^2  s ds 
\end{align}
and integrating over $\Lambda_{r_0}(x)$ gives 
\begin{align}
|\Lambda_{r_0}| v^2(x) 
&\lesssim 
\int_0^{r_0} \int_0^{\theta_0} v^2(r,\theta) r  d \theta  dr
+ |\Lambda_{r_0}| \, (\eta h)^2 \| \nabla v \|^2_{L^\infty(\Lambda_{\eta h}(x))}  
\\ 
&\qquad 
+ |\ln(d/\eta h)|  \int_0^{r_0} \int_0^{\theta_0} \left( \int_{\eta h}^{r_0} (\partial_r v(s,\theta))^2  s ds\right) 
r d \theta d r
\\
&\lesssim 
\| v \|^2_{\Lambda_{r_0}(x)} 
+ |\Lambda_{r_0}| \, (\eta h)^2 \| \nabla v \|^2_{L^\infty(\Lambda_{\eta h}(x))}  
+ d^2 |\ln(d/\eta h)| \, \| \nabla v \|^2_{\Lambda_{r_0}(x)}
\end{align}
Here $r_0 \sim 1$, and $|\Lambda_{r_0}| \sim r_0^2 \sim 1$ is independent of $x$, and 
thus we obtain
\begin{align}
v^2(x) \lesssim \| v \|^2_{\Lambda_{r_0}(x)} 
+ |\ln (d/\eta h)| \| \nabla v \|^2_{\Lambda_{r_0}(x)} 
+(\eta h)^2 \| \nabla v \|^2_{L^\infty(\Lambda_{\eta h}(x))}  
\end{align}
which leads to 
\begin{align}
\| v \|^2_{L^\infty(U_{\delta_0,\epsilon_0}(z))} 
\lesssim 
(1 + |\ln(h)|) \| v \|^2_{H^1(U_{\delta_0,\epsilon_0}(z))} 
+
h^2 \| \nabla v \|^2_{L^\infty(U_{\delta_0,\epsilon_0}(z))}
\end{align}
and thus (\ref{eq:dscrete-sobolev-a}) holds.

\paragraph{2. $\boldsymbol{d=2}$.}
In the two dimensional case $d=2$, the interface $\Sigma$ consist of a set of 
isolated points and we may cover the two dimensional set $U_{\delta_0,\epsilon_0}(z)$ 
by a patch of elements $\mcT_h(U_{\delta_0,\epsilon_0})$, and then apply the element 
wise inverse inequality (\ref{Linfty}),
\begin{align}
\| v \|^2_{L^\infty(U_{\delta_0,\epsilon_0}(z))} 
&\lesssim 
(1 + |\ln(h)|) \| v \|^2_{H^1(U_{\delta_0,\epsilon_0}(z))} 
+
h^2 \| \nabla v \|^2_{L^\infty(U_{\delta_0,\epsilon_0}(z))}
\\
&\lesssim 
(1 + |\ln(h)|) \| v \|^2_{H^1(\mcT_h(U_{\delta_0,\epsilon_0}(z)))} 
+
h^2 \| \nabla v \|^2_{L^\infty(\mcT_h(U_{\delta_0,\epsilon_0}(z)))}
\\
&\lesssim 
(1 + |\ln(h)|) \| v \|^2_{H^1(\mcT_h(U_{\delta_0,\epsilon_0}(z)))} 
+
\| \nabla v \|^2_{\mcT_h(U_{\delta_0,\epsilon_0}(z)))}
\\
&\lesssim (1 + |\ln(h)|)  \tn v \tn_h^2
\end{align}
where we finally used the stabilization estimate (\ref{eq:inverse-jumps}). 
This completes the proof in the case $d=2$.

\paragraph{3. $\boldsymbol{d\geq 3}$.} Here, the set 
$U_{\delta_0,\epsilon_0}(z)$, for a given $z \in \Sigma$, is a subset of a two 
dimensional plane, that cuts through the $d$ dimensional elements in a general 
way, which requires a more refined argument since an element wise trace inequality 
can not be applied due to the presence of cut elements. We start by integrating (\ref{eq:dscrete-sobolev-a})  over $\Sigma$,
\begin{align}
\int_\Sigma \| v \|^2_{L^\infty(U_{\delta_0,\epsilon_0}(z))} 
&\lesssim 
(1 + |\ln(h)|) \int_\Sigma \| v \|^2_{H^1(U_{\delta_0,\epsilon_0}(z))} 
+
h^2 \int_\Sigma \| \nabla v \|^2_{L^\infty(U_{\delta_0,\epsilon_0}(z))}
\\
&
\lesssim 
(1 + |\ln(h)|) \| v \|^2_{H^1(\mcT_h(U_{\delta_0,\epsilon_0}))} 
+
\| \nabla v \|^2_{\mcT_h(U_{\delta_0,\epsilon_0})}
\\
&
\lesssim 
(1 + |\ln(h)|) \| v \|^2_{H^1(\mcT_h(U_{\delta_0,\epsilon_0}))} 
\\
&\lesssim 
(1 + |\ln(h)|) \tn v \tn^2_h
\end{align}
Here we used the inverse estimate 
\begin{align}\label{eq:discrete-sobolev-b2}
h^2 \int_\Sigma \| \nabla v \|^2_{L^\infty(U_{\delta_0,\epsilon_0}(z))}
\lesssim 
\| \nabla v \|^2_{\mcT_h(U_{\delta_0,\epsilon_0})}
\end{align}
To verify (\ref{eq:discrete-sobolev-b2}) we first note that, 
with $w = \nabla v$, we have for each $z\in \Sigma$,
\begin{align}
\| w \|^2_{L^{\infty}(U_{\delta_0,\epsilon_0}(z))} 
&=\max_{T \in \mcT_h(U_{\delta_0,\epsilon_0}(z))} 
\| w \|^2_{L^{\infty}(U_{\delta_0,\epsilon_0}(z)\cap T)} 
\\
&\lesssim 
\sum_{T \in \mcT_h(U_{\delta_0,\epsilon_0}(z))} 
\| w \|^2_{L^{\infty}(U_{\delta_0,\epsilon_0}(z)\cap T)} 
\\  \label{eq:discrete-sobolev-d}
&\lesssim   
\sum_{T \in \mcT_h(U_{\delta_0,\epsilon_0})} 
\| w \|^2_{L^{\infty}(T)} 1_T(z) 
\\  \label{eq:discrete-sobolev-e}
&\lesssim   
\sum_{T \in \mcT_h(U_{\delta_0,\epsilon_0})} 
h^{-d} \| w \|^2_{T} 1_T(z) 
\end{align}
where $1_T(z) = 1$ if $U_{\delta_0,\epsilon_0}(z)\cap T \neq \emptyset$ and $0$ 
otherwise, and we employed an inverse inequality in the last step. We next note 
that $1_T:\Sigma \rightarrow \{0,1\}$ is the characteristic function of the closest 
point projection $p_\Sigma(T)$ of $T$ on $\Sigma$, and therefore 
\begin{equation} \label{eq:discrete-sobolev-f}
\int_\Sigma 1_T \lesssim h^{d-2}
\end{equation}
Integrating, (\ref{eq:discrete-sobolev-e}) over $\Sigma$ we get 
\begin{align}
\int_{\Sigma} \| w \|^2_{L^{\infty}(U_{\delta_0,\epsilon_0}(z))} 
&\lesssim 
\int_{\Sigma} \sum_{T \in \mcT_h(U_{\delta_0,\epsilon_0})} 
h^{-d} \| w \|^2_T 1_T(z) 
\\
& \lesssim 
\sum_{T\in \mcT_h(U_{\delta_0,\epsilon_0})} h^{-d} \| w \|^2_T  \int_\Sigma 1_T(z) 
\\
&= h^{-2} \| w \|^2_{\mcT_h(U_{\delta_0,\epsilon_0})}
\end{align}
where we used  (\ref{eq:discrete-sobolev-f}). This completes the verification of (\ref{eq:discrete-sobolev-b2}).

%let, for each element 
%$T$ intersected by a $2$ dimensional plane $S$, $\text{Cyl}(T) = B_r \times I^{d-2} 
%\subset \IR^d$, where $B_r$ with $r\sim h$ an open disc in $S$ and $I$ is an open 
%interval of length $|I| \sim h$, be a cylinder containing $T$. We then have, with 
%$w = \nabla v$, the inverse estimate 
%\begin{align}
%\int_{z \in I^{d-2}} \| w \|^2_{L^\infty(T \cap S(z) )} 
%&\lesssim 
% \int_{z \in I^{d-2} } \| w \|^2_{L^\infty(\text{Cyl}(T) \cap S(z) )}
%\\ 
%&\qquad \lesssim \label{eq:discrete-sobolev-c}
% h^{d-2} \| w \|^2_{L^\infty(\text{Cyl}(T))}
%  \lesssim 
%h^{-2} \| w \|^2_{\text{Cyl}(T)}
%  \lesssim 
% h^{-2}  \| w \|^2_T
%\end{align}
%where $w$ on $\text{Cyl}(T)$ is the polynomial extension of $w|_T$. We also note that 
%
%
%Integrating (\ref{eq:discrete-sobolev-d}) over $\Sigma$ and using the inverse bound (\ref{eq:discrete-sobolev-c}) we get 
%\begin{align}
%\int_{\Sigma} \| w \|^2_{L^{\infty}(U_{\delta_0,\epsilon_0}(z))} 
%&\lesssim 
%\int_{\Sigma} \sum_{T \in \mcT_h(U_{\delta_0,\epsilon_0})} 
%\| w \|^2_{L^{\infty}(T)} 1_T(z) 
%\\
%&\lesssim 
%\sum_{T\in \mcT_h(U_{\delta_0,\epsilon_0})} 
% \int_{z \in I^{d-2} } \| w \|^2_{L^\infty(\text{Cyl}(T) \cap S(z) )} 1_T(z) 
%\\
%& \lesssim 
%\sum_{T\in \mcT_h(U_{\delta_0,\epsilon_0})} h^{-2} \| w \|^2_T  1_T(z) 
%\\
%&= h^{-2} \| w \|^2_{\mcT_h(U_{\delta_0,\epsilon_0})}
%\end{align}
%which completes the proof of (\ref{eq:discrete-sobolev-b2}). 
%%\todo{$w$ or $\nabla v$?}
\end{proof}

\begin{lem}  \label{lem:weight_stab} Let $\chi$ be defined by (\ref{eq:cutoff}), then there is a constant 
such that for all $v\in V_h$,
\begin{empheq}[box = \widefbox]{equation}\label{eq:cutoff-v-H1}
\| (\nabla \chi) v \|_{U_{\delta,\epsilon}} 
\lesssim 
( 1 + |\ln(h)| )  \tn v \tn_{h}
\end{empheq}
%
%\begin{empheq}[box = \widefbox]{equation}\label{eq:cutoff-v-H1}
%\| (\nabla \chi) v \|^2_{U_{\delta,\epsilon}} 
%\lesssim 
%|\ln(1+\delta/\epsilon)|\delta^{-2} \| v \|^2_{\mcT_h(U_{\delta,\epsilon})}
%\lesssim 
%|\ln(1+\delta/\epsilon)| \, \tn v \tn^2_{1,h}
%\end{empheq}
%\todo[inline]{Form of log factor ? Check also power of the log factor}
\end{lem}
\begin{proof} Splitting $\| (\nabla \chi) v \|^2_{U_{\delta,\epsilon}}$ into three 
contributions corresponding to the directions of the derivative relative to the 
interface $\Sigma$ we obtain
\begin{align}
\| (\nabla \chi) v \|^2_{U_{\delta,\epsilon}} 
&\lesssim
\| (\nabla_\Sigma \chi) v \|^2_{U_{\delta,\epsilon}} 
+ \| (\nabla_n \chi) v \|^2_{U_{\delta,\epsilon}} 
+ \| (\nabla_\nu \chi) v \|^2_{U_{\delta,\epsilon}} 
\\ \label{eq:second_line}
&\lesssim \| v \|^2_{U_{\delta,\epsilon}} 
+ \delta^{-2} \| v \|^2_{U_{\delta,\epsilon}} 
+ \| (\nabla_\nu \chi) v \|^2_{U_{\delta,\epsilon}}
\\ 
&\lesssim \| v \|^2_{U_{\delta,\epsilon}} 
+
\int_\Sigma \| v \|^2_{L^\infty(U_{\delta,\epsilon}(z))}  
+ 
(1 + |\ln(h)|)^2  \tn v \tn^2_{h}
\\
&\lesssim (1 + |\ln(h)|^2) \, \tn v \tn^2_{h}
\end{align}
where we for the second term \eqref{eq:second_line} used the facts 
$|U_{\delta,\epsilon}(z))| \lesssim \delta^2 \lesssim h^2$,  
$\| v \|_{L^\infty(U_{\delta,\epsilon}(z))} \leq  \| v \|_{L^\infty(U_{\delta_0,\epsilon_0}(z))}$ 
followed by (\ref{eq:dscrete-sobolev}), and for the third term we used the 
estimate 
\begin{equation}\label{eq:tech-cutoff-lem-a}
\| (\nabla_\nu \chi) v \|_{U_{\delta,\epsilon}} \lesssim (1 +  |\ln(h)| ) \tn v \tn_{h}
\end{equation}
which we verify next. This argument completes the proof of (\ref{eq:cutoff-v-H1}).  

To verify (\ref{eq:tech-cutoff-lem-a}) we use H\"older's inequality twice, first on 
$U_{\delta,\epsilon}(z)$ and then on $\Sigma$, employ (\ref{eq:cutoff-bound-b}), 
and finally (\ref{eq:dscrete-sobolev}),
\begin{align}
\|(\nabla_\nu \chi) v \|^2_{U_{\delta,\epsilon}} 
&=\int_{\Sigma} \|(\nabla_\nu \chi) v \|^2_{U_{\delta,\epsilon}(z)} 
\\ 
&\lesssim 
\int_{\Sigma} \| \nabla_\nu \chi \|^2_{U_{\delta,\epsilon}(z)} 
\| v \|^2_{L^{\infty}(U_{\delta,\epsilon}(z))} 
\\ 
&\lesssim 
\Big( \sup_{z\in \Sigma}  \| \nabla_\nu \chi \|^2_{U_{\delta,\epsilon}(z)} \Big) 
\int_{\Sigma} \| v \|^2_{L^{\infty}(U_{\delta,\epsilon}(z))} 
\\ 
&\lesssim 
(1 + |\ln(h)|) \int_{\Sigma} \| v \|^2_{L^{\infty}(U_{\delta,\epsilon}(z))} 
\\ 
&\lesssim 
(1 + |\ln(h)|) \int_{\Sigma} \| v \|^2_{L^{\infty}(U_{\delta_0,\epsilon_0}(z))} 
\\ 
&\lesssim
(1 + |\ln(h)|)^2 \tn v \tn^2_{h}
\end{align}
Thus (\ref{eq:tech-cutoff-lem-a}) holds.
\end{proof}
\begin{lem} There is a constant such that for all $w \in V_h$,
\begin{equation}\label{eq:technical-aa}
h^{-2} \|w\|^2_{\mcT_h(U_\delta)} 
+ 
\|\nabla w\|^2_{\mcT_h(U_\delta)} 
\lesssim 
\tn w \tn^2_{h}
\end{equation}
which holds for $\delta = \eta h$ with $\eta$ a sufficiently small constant. 
\end{lem}
\begin{proof} 
First observe that by construction no point in $U_\delta$ is further 
than $O(\delta)$ from $\partial \Omega_D$. Using estimate 
\eqref{eq:L2inverse_jumps} followed by the Poincar\'e inequality 
\begin{align}
\|w\|^2_{\mcT_h(U_\delta(\partial \Omega_D))} 
&\lesssim \delta \|w\|_{\partial \Omega_D} +
\delta^2 \|\nabla w\|_{\mcT_h(U_\delta(\partial \Omega_D))}
%\\
%&\lesssim h \|w\|_{\partial \Omega_D} +
%h^2 \|\nabla w\|_{\mcT_h(U_\delta(\partial \Omega_D))}
\end{align}
see appendix \cite{BHL18}, we obtain
\begin{align}
\|w\|^2_{\mcT_h(U_\delta)} 
&\lesssim \|w\|^2_{\mcT_h(U_\delta(\partial \Omega_D))} 
+ \|w\|^2_{\mcT_h(U_{\delta,\epsilon})} 
\\
&
\lesssim \|w\|^2_{\mcT_h(U_\delta(\partial \Omega_D))} + 
h^3 \|[\nabla_n w]\|^2_{\mcF_h(\partial \Omega \cap U_\delta)}
\\
&
\lesssim  \delta \|w\|_{\partial \Omega_D} +
\delta^2 \|\nabla w\|_{\mcT_h(U_\delta(\partial \Omega_D))}
  + h^2\|\nabla w\|^2_{\mcT_h(\partial \Omega \cap U_\delta)} 
%+ \sum_{F \in \mcF_h(\partial \Omega \cap U_\delta)} \textcolor{black}{h^3}
%\|[\nabla_n w]\|^2_F
\end{align}
where we used the estimate
\begin{equation}
h \|[\nabla_n w]\|^2_{\mcF_h(\partial \Omega \cap U_\delta)} \lesssim \| \nabla v \|^2_{\mcT_h} 
\end{equation}
Applying now \eqref{eq:inverse-jumps} and using $\delta \sim h$ we conclude that
\begin{equation}
h^{-2} \|w\|^2_{\mcT_h(U_\delta)} 
+   \|\nabla w\|^2_{\mcT_h(U_\delta)} 
\lesssim h^{-1} \|w\|^2_{\partial \Omega_D}
+\|\nabla w\|^2_{\Omega}+ \| w \|^2_{s_h} 
\lesssim \tn w \tn^2_{h}
\end{equation}
\end{proof}

\begin{lem}\label{lem:trace-cont} There 
is a constant such that for all $v \in V,$ $v_h \in V_h,$ and $w \in V_h$,
\begin{empheq}[box = \widefbox]{align}\nonumber
(\nabla_n (v-v_h),w)_{\chi,\partial \Omega } &\lesssim
  \Big( 
(1 + |\ln(h)| ) \| \nabla (v-v_h) \|_{U_\delta} 
\\ \label{eq:trace-cont-b}
&\qquad 
+ h \|\Delta v \|_{U_\delta}  
 + h^{1/2} \|[\nabla_n v_h]\|_{\mcF_h\cap U_\delta}\Big)  \tn w \tn_{h}
\end{empheq}
\end{lem}
\begin{proof} For $v \in V$, see (\ref{eq:defV}), we have 
$\Delta v \in L^2( \text{supp}(\chi)) \subset L^2(U_{\delta_0})$ and  using Green's formula 
\begin{align}
(\Delta v, \chi w )_{\Omega} 
= (\nabla_n v, \chi w )_{\partial \Omega} - ( \nabla v, (\nabla \chi)  w)_\Omega 
- ( \nabla v, \chi  \nabla w)_\Omega 
\end{align}
%and thus 
%\begin{align}
% (\nabla_n v, \chi w )_{\partial \Omega} =
% (\Delta v, \chi w )_{\Omega} + ( \nabla v, (\nabla \chi)  w)_\Omega 
%+ ( \nabla v, \chi  \nabla w)_\Omega 
%\end{align}
For $v_h \in V_h$ we use Green's formula element wise
\begin{align}
(\nabla v_h, \chi \nabla w)_{\Omega} 
&=(\nabla_n v_h, \chi w)_{\partial \Omega} 
+ ([\nabla_n v_h],\chi w)_{\mcF_h\cap \Omega} 
\\ 
&\qquad 
-(\Delta v_h, \chi w)_{\mcT_h \cap \Omega} 
- (\nabla v_h, (\nabla \chi) w)_\Omega  
 \end{align}
Combining the formulas and rearranging the terms we obtain
\begin{align}
(\nabla_n (v-v_h), w)_{\chi,\partial \Omega}
&=
(\nabla  (v-v_h), \chi \nabla w)_{\Omega}  
+ 
{(\nabla  (v-v_h), (\nabla \chi) w)_\Omega}
\\ 
&\qquad + 
(\Delta v, \chi w)_\Omega 
+
([\nabla_n v_h],\chi w)_{\mcF_h\cap \Omega} 
\end{align}
To estimate the right hand side we may directly estimate the first 
two terms using the Cauchy Schwarz inequality and (\ref{eq:technical-aa}),
\begin{gather}
(\nabla (v-v_h), \chi \nabla w)_{\Omega}  
\lesssim \| \nabla (v-v_h) \|_{U_\delta} \| \nabla w \|_{U_\delta}
\lesssim \| \nabla (v-v_h) \|_{U_\delta} \tn w \tn_{h}
\\[3mm] 
(\Delta v, \chi w)_\Omega 
\lesssim h \| \Delta v \|_{U_\delta} h^{-1} \|w\|_{U_\delta} 
\lesssim h \| \Delta v\|_{U_\delta} \tn w \tn_{h}
\end{gather}
Next using the Cauchy Schwarz inequality, the element wise trace 
inequality (\ref{trace_H1}), 
\begin{align}
([\nabla_n v_h],\chi w)_{\mcF_h\cap \Omega} 
%&\lesssim
%\delta^{1/2} \|[\nabla_n v_h]\|_{\mcF_h\cap U_\delta} \delta^{-1/2} \|w\|_{\mcF_h\cap U_\delta} 
%\\
&\lesssim
h^{1/2} \|[\nabla_n v_h]\|_{\mcF_h\cap U_\delta} 
h^{-1/2}( h^{-1} \|w\|^2_{\mcT_h(U_\delta)} + h \|\nabla w\|^2_{\mcT_h(U_\delta)} )^{1/2}
\\
&\lesssim
h^{1/2} \|[\nabla_n v_h]\|_{\mcF_h\cap U_\delta} 
( h^{-2} \|w\|^2_{\mcT_h(U_\delta)} + \|\nabla w\|^2_{\mcT_h(U_\delta)} )^{1/2}
\\
&\lesssim
h^{1/2} \|[\nabla_n v_h]\|_{\mcF_h\cap U_\delta} \tn w \tn_{h}
\end{align}
where for the last inequality we employed (\ref{eq:technical-aa}). For the remaining term 
we use the Cauchy Schwarz inequality, followed by  (\ref{eq:technical-aa}) and (\ref{eq:cutoff-v-H1}),
 \begin{align}
(\nabla (v-v_h), (\nabla \chi) w)_\Omega 
&\lesssim 
\| \nabla (v-v_h) \|_{U_\delta} \Big( \| (\nabla \chi) w \|_{U_\delta(\partial \Omega_D)} 
+ \| (\nabla \chi) w \|_{U_{\delta,\epsilon}} \Big)
\\
&\lesssim 
\| \nabla (v-v_h) \|_{U_\delta} 
\Big( \delta^{-1} \| w \|_{U_\delta(\partial \Omega_D)} 
+   \| (\nabla \chi) w \|_{U_{\delta,\epsilon}}  \Big)
\\
&\lesssim 
(1 + |\ln(h)| ) \| \nabla (v-v_h) \|_{U_\delta}  \tn w \tn_{h}
\end{align}
Collecting the bounds we arrive at 
\begin{align}
(\nabla_n (v-v_h), w)_{\partial \Omega}
 &\lesssim  \Big(  ( 1 +|\ln(h)|) 
\| \nabla (v-v_h) \|_{U_\delta} 
\\
&\qquad
+ h \| \Delta v \|_{U_\delta}
+ h^{{1/2}} \|[\nabla_n v_h ] \|_{\mcF_h\cap U_\delta} \Big) \tn w \tn_{h}
\end{align}
which completes the proof of (\ref{eq:trace-cont-b}).

%Using the bounds together with the 
%Cauchy Schwarz inequality we arrive at
%\begin{align} 
%(\nabla_n(v-v_h), w)_{\partial \Omega}
%&=
%(\nabla (v-v_h), \chi \nabla w)_{\Omega}  
%+ 
%{(\nabla (v-v_h), (\nabla \chi) w)_\Omega}
%\\ 
%&\qquad + 
%(\Delta v, \chi w)_\Omega 
%-
%([\nabla_n v_h],\chi w)_{\mcF_h\cap \Omega} 
%\\ 
%&\lesssim 
%\Big(|\ln(h)| \|\nabla (v-v_h)\|^2_{U_\delta} 
%+ 
%\delta^2 \|\Delta v\|^2_{U_\delta} 
%+
%\delta \|[\nabla_n v_h]\|_{\mcF_h\cap U_\delta}
%\Big)^{1/2}
%\\ 
%&\qquad 
%\times 
%{{\Big( \delta^{-2}(1 + \delta h^{-1} ) \|w\|^2_{\mcT_h(U_\delta)} 
%+ 
%(1+ \delta^{-1} h) \|\nabla w\|^2_{\mcT_h(U_\delta)} \Big)   }}^{1/2}
%\\ 
%&\lesssim 
%\Big(|\ln(h)| \|\nabla (v-v_h)\|^2_{U_\delta} 
%+ 
%\delta^2 \|\Delta v\|^2_{U_\delta} 
%+
%\delta \|[\nabla_n v]\|_{\mcF_h\cap U_\delta}
%\Big)^{1/2} 
%\tn w \tn_h
%\end{align}

\end{proof}

\subsection{Interpolation}

Let $E:H^s(\Omega) \rightarrow H^s(\IR^d)$ be a continuous extension operator. 
Define the interpolant $\pi_h : H^1(\Omega) \rightarrow V_h$ by $\pi_h = \pi_{h,Cl} \circ E$ 
where $\pi_{h,Cl}: L^2(\Omega_h) \rightarrow V_h$ is the Clement interpolant and 
$\Omega_h = \cup_{T\in \mcT_h} T$. Using the interpolation results for the Clement 
interpolation operator and the stability of the extension operator we conclude that
\begin{equation}\label{eq:interpol}
\| v - \pi_h v \|_{H^m(\Omega)} \lesssim h^{s-m} \| v \|_{H^s(\Omega)}\qquad 0 \leq m \leq s \leq 2
\end{equation} 
%\todo{Check extension for Lipschitz boundary}
For the energy norm (\ref{eq:energy-norm-nostab}) it holds
\begin{empheq}[box = \widefbox]{equation}\label{eq:interpol-energy}
\tn v - \pi_h v \tn + \|\pi_h v\|_{s_h} \lesssim h^{s-1} \| v \|_{H^s(\Omega)}
\end{empheq}
\begin{proof} With $\rho = v - \pi_h v$ we have  
\begin{equation}
\tn \rho \tn^2_{0,h} \lesssim \| \nabla \rho \|^2_\Omega + h^{-1} \| \rho \|^2_{\bd}
\end{equation}
Using (\ref{eq:interpol}) we directly have 
\begin{equation}
\| \nabla \rho \|^2_\Omega \lesssim h^{2(s-1)} \| u \|^2_{H^s(\Omega)}
\end{equation}
and using the trace inequality 
\begin{equation}
\| v \|^2_{\bd} \lesssim \delta^{-1} \| v \|^2_{U_\delta(\bd)} 
+ \delta \| \nabla v \|^2_{U_\delta(\bd)}  
\end{equation}
with $\delta \sim h$ we obtain
\begin{align}
h^{-1} \| \rho \|^2_{\bd}
&\lesssim 
h^{-1} ( \delta^{-1} \| \rho \|^2_{U_\delta(\bd)} + \delta  \| \nabla \rho \|^2_{U_\delta(\bd)} )
\\
&\lesssim   h^{-2}\| \rho \|^2_{U_\delta(\bd)}  +  \|  \nabla \rho \|^2_{U_\delta(\bd)}  
\\
&\lesssim h^{2(s-1)} \| v \|^2_{H^s(\Omega)}
\end{align}
Finally, we have with $\pi_{h,Cl} \nabla E v \in V_h^d$,
\begin{align}
\|\pi_h v\|^2_{s_h} &\lesssim  h \| [\nabla \pi_h v -
  \pi_{h,Cl} \nabla E v] \|^2_{\mcF_h} \\
& \lesssim \|\nabla \pi_h v -
  \pi_{h,Cl} \nabla E v\|^2_{\mcT_h}\\
&\lesssim  \|\nabla_n (\pi_h v -v)\|^2_{\mcT_h}
  + \|\pi_{h,Cl} \nabla E v - \nabla Ev\|^2_{\mcT_h}
\lesssim   h^{2(s-1)} \| v \|^2_{H^s(\Omega)}
\end{align}
In the first inequality the inverse inequality 
\begin{equation}
h \| [\nabla w ] \|^2_F \lesssim \| \nabla w \|^2_{T_1} +  \| \nabla w \|^2_{T_2}, \qquad 
w \in V_h|_{T_1 \cup T_2}
\end{equation}
where $T_1$ and $T_2$ are the two 
elements that share face $F$.
\end{proof}

\subsection{Error Estimates}

\begin{thm} Let $u \in H^s(\Omega)$, $s \in [1,3/2]$, be the solution to 
(\ref{eq:strong-mixed-a})-(\ref{eq:strong-mixed-b}) and $u_h$ the finite 
element approximation defined by (\ref{eq:method}), then 
\begin{empheq}[box = \widefbox]{align}\nonumber
\tn u - u_h \tn + \|u_h\|_{s_h} &\lesssim h^{s-1} \Big( (1 + |\ln(h)|) 
\| u \|_{H^s(\Omega)} + \| g_N \|_{\widetilde{H}^{s-3/2}(\bn)} \Big)
\\  \nonumber
&\qquad + h \Big( \| f \|_{U_\delta}  + \| f \|_{H^{-1}(\Omega)} + \| g_N \|_{\widetilde{H}^{-1/2}(\bn)}  
 + \| g_D \|_{H^{1/2}(\bd)} \Big)
\end{empheq}
The logarithmic factor is present only for the case of
mixed Dirichlet-Neumann boundary conditions.
\end{thm}
\begin{proof}
We split the error as follows
\begin{align}\nonumber
\tn u - u_h \tn +  \|u_h\|_{s_h}
&\lesssim \tn u - \pi_h u \tn_{h} + \tn \pi_h u - u_h \tn_{h} + \|u_h\|_{s_h}
\\ \nonumber
&\lesssim 
\underbrace{\tn u - \pi_h u \tn_{h}}_{\lesssim h^{s-1} \|u \|_{H^s(\Omega)}} 
+ \underbrace{\tn \pi_h u - u_{h,\epsilon} \tn_{h}}_{I} 
+ \underbrace{\tn u_{h,\epsilon} - u_h \tn_{h}}_{II} +\underbrace{ \|u_h\|_{s_h}}_{III}
\end{align}
where $u_{h,\epsilon}$ is the solution to the regularized 
problem (\ref{eq:method-eps}) and we used the interpolation 
error estimate (\ref{eq:interpol-energy}) to estimate the first 
term on the right hand side.

\paragraph{Term $\bfI$.}  The following estimate holds
\begin{equation}\label{eq:error-est-I}
\tn \pi_h u - u_{h,\epsilon} \tn_{h} 
\lesssim 
(1 + |\ln(h)|) h^{s-1} \Big( \| u \|_{H^s(\Omega)}
+   \| g_N \|_{\widetilde{H}^{s-3/2}(\bn)} \Big) + h\|f\|_{U_\delta}
\end{equation}
To verify the estimate let $\rho_h = \pi_h u - u_{h,\epsilon}$. Using coercivity (\ref{eq:coercivity}) we obtain
\begin{equation}\nonumber
\tn \rho_h \tn_h^2
\lesssim 
A_{h,\epsilon}(\rho_h  ,\rho_h )+s_h(\rho_h ,\rho_h)
\end{equation}
and then employing the definition (\ref{eq:method-eps}) of $u_{h,\epsilon}$ we 
obtain
\begin{align}
&A_{h,\epsilon}(\pi_h u - u_{h,\epsilon} ,\rho_h) +s_h(\pi_h u - u_{h,\epsilon} ,\rho_h)
\\
&\qquad =
A_{h,\epsilon}(\pi_h u ,\rho_h) - L_h(\rho_h) +s_h(\pi_h u ,\rho_h)
\\ 
&\qquad =
A_{h,\epsilon}(\pi_h u - u, \rho_h) 
+ 
A_{h,\epsilon}(u,\rho_h) - L_h(\rho_h) +s_h(\pi_h u ,\rho_h)
\\ \label{eq:error-est-a}
&\qquad \lesssim 
(\tn \pi_h u - u \tn + \|\pi_h u\|_{s_h})\tn \rho_h \tn_h 
+ 
|(\nabla_n( \pi_h u - u ), \rho_h )_{\chi,\partial \Omega}|
\\ 
&\qquad \qquad 
+
|A_{h,\epsilon}(u,\rho_h) - L_h(\rho_h)|
\\ \label{eq:error-est-b}
&\qquad \lesssim 
h^{s-1} \| u \|_{H^s(\Omega)} \tn \rho_h \tn_h 
+
(1+ |\ln(h)|) h^{s-1} \| u \|_{H^s(\Omega)} + h \|f\|_{U_\delta})\tn \rho_h \tn_h 
\\ 
&\qquad \qquad + 
(1 + |\ln(h)|) | h^{s-1}  \| g_N \|_{\widetilde{H}^{s-3/2}(\bn)}  \tn \rho_h \tn_h 
\end{align}
%\todo{!!}
where we used the continuity (\ref{eq:continuity}) in (\ref{eq:error-est-a}), and 
in (\ref{eq:error-est-b}) we used the interpolation error estimate 
(\ref{eq:interpol-energy}) to estimate the first term and then the following estimates 
\begin{gather}\label{eq:error-est-c}
|(\nabla_n( \pi_h u - u ), \rho_h )_{\chi,\partial \Omega}| 
\lesssim \Big( (1+|\ln(h)|)  h^{s-1} \| u \|_{H^s(\Omega)} + h \|f\|_{U_\delta} \Big)\tn  \rho_h  \tn_h
\\ \label{eq:error-est-d}
|A_{h,\epsilon}(u, \rho_h ) - L_h( \rho_h )|
\lesssim
(1+|\ln(h)|)  h^{s-1}  \| g_N \|_{\widetilde{H}^{s-3/2}(\bn)}  \tn \rho_h \tn_h 
\end{gather}

\paragraph{(\ref{eq:error-est-c}).} Using (\ref{eq:trace-cont-b}) followed by the 
interpolation estimate (\ref{eq:interpol-energy}),
%\todo{ii}
\begin{align}
&|(\nabla_n( \pi_h u - u ),\rho_h )_{\chi,\partial \Omega}| 
\\
&\qquad  \lesssim \Big(
(1 + |\ln(h)|)\| \nabla(u - \pi_h u ) \|_{U_\delta} 
\\
&\qquad \qquad \qquad 
+ h \| \Delta u \|_{U_\delta}  
+ h^{1/2} \| [\nabla_n \pi_h u]\|_{\mcF_h \cap U_\delta}  
) 
\Big)
\tn \rho_h \tn_h 
\\ 
&\qquad \lesssim 
\Big( (1+|\ln(h)|) h^{s-1} \| u  \|_{H^s(\Omega)} + h \| f \|_{U_\delta} ) \Big) \tn \rho_h\tn_h
\end{align}
where we used the fact $\Delta u = -f$. 
%\todo{check continuity}
\paragraph{(\ref{eq:error-est-d}).} Starting from the identity 
(\ref{eq:residual-regularized}) we get
\begin{align}
|A_{h,\epsilon}(u,\rho_h) - L_h(\rho_h)| &= |(g_N, \chi \rho_h)_{\bn}| 
\\
&\lesssim 
 \| g_N \|_{\widetilde{H}^{s-3/2}(\bn)} 
\| \chi \rho_h\|_{H^{3/2-s}(\bn)}
\end{align}
%
%\begin{align}
%A_{h,\epsilon}(u,v) - L_h(v) 
%&=-(\Delta u, v)_\Omega 
%+ (\nabla_n u,v)_{\partial \Omega} 
%- (\nabla_n u, \chi v)_{\partial \Omega}
%\\ 
%&\qquad \qquad 
%-(u, \chi \nabla_n v)_{\bd}
%+\beta h^{-1} ( u, \chi v )_{\bd}
%\\
%&\qquad 
%- (f,v)_\Omega 
%- (g_N,v)_{\bn} 
%+ (g_D, \nabla_n v)_{\bd}
%- \beta h^{-1}(g_D, v)_{\bd} 
%\\
%&=- (g_N, \chi v)_{\bn} 
%+ \beta h^{-1} (u, \chi v)_{\bn} 
%-(u,\chi \nabla_n v )_{\bn}
%\end{align}
%\todo{!!}
To estimate 
$ \| \chi \rho_h\|_{H^{3/2-s}(\bn)}$ we use a trace inequality on $U_{\delta_0}(\bn))$,
%
%= \inf_{w \in H^{2-s}(\Omega), \text{$w= \chi \rho$ on $\bn$}} \| w \|_{H^{2-s}(\Omega)} \lesssim  \| \chi \rho_h\|_{H^{2-s}(\Omega)}$ we first note that $\text{supp}(\chi) \subset U_\delta = U_\delta(\bd) \cup U_{\delta,\epsilon}$ and thus
\begin{align}
 \| \chi \rho_h\|_{H^{3/2-s}(\bn)}
& \lesssim 
  \|\chi \rho_h \|_{H^{2-s}(U_{\delta_0}(\bn))} 
%\\
%& =
%  \| \chi \rho_h \|^2_{H^{2-s}(U_\delta(\bd))} 
%  + \|\chi \rho_h  \|^2_{H^{2-s}(U_{\delta,\epsilon}(\bd))}
\end{align}
In order to estimate the right hand side using the available bounds we employ 
the interpolation between norms estimate
\begin{equation}
\| v \|_{H^\gamma(\omega)} \lesssim \| v \|^{1-t}_{H^{s_1}(\omega)} 
\| v \|^{t}_{H^{s_2}(\omega)}
\end{equation}
for $t \in [0,1]$ and $\gamma = (1-t) s_1 + t s_2$. In our case $\gamma = 2-s \in [1/2,1]$
and we take $s_1 = 0$ and $s_2 =1$, which gives $t = 2-s$. 
%
%
%We obtain using 
%the fact that $\| \nabla \chi \|_{L^\infty(U_\delta(\bd))} \lesssim \delta^{-1}$, 
%see (\ref{eq:cutoff}), and (\ref{eq:technical-aa}),
%\begin{align}
% \| \chi \rho_h \|^2_{H^{2-s}(U_\delta(\bd))} 
% &\lesssim 
%\| \chi \rho_h \|^{s-1}_{H^0(U_{\delta}(\bd))} \| \chi \rho_h \|^{2-s}_{H^1(U_{\delta}(\bd))}
%\\
%&\lesssim
%\Big( h \tn \rho_h \tn_h \Big)^{s-1} 
%\Big( (h + h \delta^{-1})\tn \rho_h  \tn_h \Big)^{2-s} 
%\\
%&\lesssim 
%h^{s-1} \tn \rho_h \tn_h
%\end{align}
Observing that $\text{supp}(\chi) \cap U_{\delta_0} (\bn) \subset U_{\delta,\epsilon}$ we get 
\begin{align}
\| \chi \rho_h \|_{H^{2-s}(U_{\delta,\epsilon})}
&\lesssim 
\| \chi \rho_h \|^{s-1}_{H^0(U_{\delta,\epsilon})} \| \chi \rho_h \|^{2-s}_{H^1(U_{\delta,\epsilon})}
\\
%&\lesssim 
%\| \rho_h \|^{s-1}_{U_{\delta,\epsilon}} (1 +|\ln(h)|)^{2-s} \tn \rho_h \tn^{2-s}
%\\
&\lesssim \Big( (1 +|\ln(h)|) h \tn \rho_h \tn_h\Big)^{s-1} \Big( (1 +|\ln(h)|) \tn \rho_h \tn_h\Big)^{2-s}
\\
&\lesssim (1 +|\ln(h)|) h^{s-1} \tn \rho_h \tn_h
\end{align}
Here we used the following two estimates. First
\begin{align}
\| \rho_h \|^2_{U_{\delta,\epsilon}} 
&= \int_\Sigma \| \rho_h \|^2_{U_{\delta,\epsilon}(z)} 
\\
&\lesssim \int_\Sigma h^2 \| \rho_h \|^2_{L^\infty(U_{\delta,\epsilon}(z))} 
\\
&\lesssim \int_\Sigma h^2 \| \rho_h \|^2_{L^\infty(U_{\delta_0,\epsilon_0}(z))} 
\\
&\lesssim h^2 (1+ |\ln(h)|) \tn \rho_h \tn^2_{1,h} 
\end{align} 
where we at last used (\ref{eq:dscrete-sobolev}). Second
\begin{align}
\| \chi \rho_h \|_{H^1(U_{\delta,\epsilon})} 
&\lesssim
\| \chi \rho_h \|_{U_{\delta,\epsilon}} 
+ 
\| (\nabla \chi) \rho_h \|_{U_{\delta,\epsilon}} 
+
\| \chi \nabla \rho_h \|_{U_{\delta,\epsilon}} 
\\
&\lesssim 
(1+|\ln(h)|) \tn \rho_h \tn_h 
\end{align}
where we used (\ref{eq:cutoff-v-H1}) and (\ref{eq:technical-aa}). This 
completes the bound for Term $I$.

\paragraph{Term $\bfI\bfI $.} For $\epsilon\sim h^\alpha$ with $\alpha = d$, we shall prove 
the estimate
\begin{equation}\label{eq:error-est-II}
\tn u_{h,\epsilon} - u_h \tn_h  \lesssim h  \Big( \| f \|_{H^{-1}(\Omega)} + \| g_N \|_{\widetilde{H}(\bn)}  
 +  \| g_D \|_{\bd} \Big) 
\end{equation}
We start once again with coercivity, this time of $A_h+s_h$, using the notation 
$\zeta_h = u_{h,\epsilon} - u_h$ we have
\begin{align}
\tn \zeta_h \tn_h^2 &\lesssim A_{h}(\zeta_h,\zeta_h)+s_h(\zeta_h,\zeta_h)
\end{align}
Then using the definition of the method and estimate (\ref{eq:form-error}) 
we obtain
\begin{align}
\tn \zeta_h \tn_h^2 & \lesssim A_{h}(u_{h,\epsilon} - u_h,\zeta_h) + s_h (u_{h,\epsilon} - u_h,\zeta_h)
\\
& =
A_{h}(u_{h,\epsilon},\zeta_h) + s_h (u_{h,\epsilon},\zeta_h)-L_h(\zeta_h ) 
\\ 
& =
A_h(u_{h,\epsilon},\zeta_h) - A_{h,\epsilon}(u_{h,\epsilon},\zeta_h)
\\ 
& \lesssim 
\epsilon h^{1-d} \tn u_{h,\epsilon} \tn_h \tn \zeta_h\tn_h
\\ 
& \lesssim 
 h^{\alpha + 1-d} \Big( \| f \|_{H^{-1}(\Omega)} + \| g_N \|_{\widetilde{H}^{-1/2}(\bn)}  
 + h^{-1/2} \| g_D \|_{H^{1/2}(\bd)} \Big) \tn \zeta_h\tn_h
\\ 
& \lesssim 
 h  \Big(  \| f \|_{H^{-1}(\Omega)} + \| g_N \|_{\widetilde{H}^{-1/2}(\bn)}  
 + \| g_D \|_{H^{1/2}(\bd)} \Big)   \tn \zeta_h\tn_h
\end{align}
for $\alpha = d $, where we used the stability estimate (\ref{eq:discrete-stab-bound}).

\paragraph{Term $\bfI\bfI \bfI$.}
We finally have the following estimate for the stabilization term
\begin{align}
\|u_h\|_{s_h} & \leq \|u_h - u_{h,\epsilon} \|_{s_h}+\|\pi_h u -
              u_{h,\epsilon}\|_{s_h}+\|\pi_h u\|_{s_h}
              \\ \label{eq:error-est-III}
&= \tn \zeta_h\tn_h+\tn \rho_h\tn_h+\|\pi_h u \|_{s_h}
\end{align}
where the first two terms are estimated in (\ref{eq:error-est-II}) and (\ref{eq:error-est-I}) and the third by the interpolation estimate (\ref{eq:interpol-energy}).

\paragraph{Conclusion.} The theorem now follows by collecting the bounds for the terms 
$I$, $II$, and $III$.
\end{proof}
\textcolor{black}{
\begin{rem}
Observe that the logarithmic factor can be traced to Lemma
\ref{lem:sobolev_sigma}, Lemma \ref{lem:weight_stab} and \eqref{eq:error-est-d} all of which
are invoked only for the case of mixed boundary conditions
\end{rem}}
%\todo[inline]{add a more detailed stability bound (\ref{eq:discrete-stab-bound})}
%\begin{figure}
%\centering
%\includegraphics[scale=0.17]{./figs/sketch.jpg}
%\end{figure}
\bibliographystyle{abbrv}
\bibliography{ref-low-reg.bib}

\begin{thebibliography}{1}

\bibitem{BCHLM15}
E.~Burman, S.~Claus, P.~Hansbo, M.~G. Larson, and A.~Massing.
\newblock Cut{FEM}: discretizing geometry and partial differential equations.
\newblock {\em Internat. J. Numer. Methods Engrg.}, 104(7):472--501, 2015.

\bibitem{BH12}
E.~Burman and P.~Hansbo.
\newblock Fictitious domain finite element methods using cut elements: {II}.
  {A} stabilized {N}itsche method.
\newblock {\em Appl. Numer. Math.}, 62(4):328--341, 2012.

\bibitem{BHL18}
E.~Burman, P.~Hansbo, and M.~G. Larson.
\newblock A cut finite element method with boundary value correction.
\newblock {\em Math. Comp.}, 87(310):633--657, 2018.

\bibitem{BHLZ16}
E.~Burman, P.~Hansbo, M.~G. Larson, and S.~Zahedi.
\newblock Cut finite element methods for coupled bulk-surface problems.
\newblock {\em Numer. Math.}, 133(2):203--231, 2016.

\bibitem{DiPE12}
D.~A. Di~Pietro and A.~Ern.
\newblock {\em Mathematical aspects of discontinuous {G}alerkin methods},
  volume~69 of {\em Math\'ematiques \& Applications (Berlin) [Mathematics \&
  Applications]}.
\newblock Springer, Heidelberg, 2012.

\bibitem{EG18}
A.~Ern and J.-L. Guermond.
\newblock Analysis of the edge finite element approximation of the {M}axwell
  equations with low regularity solutions.
\newblock {\em Comput. Math. Appl.}, 75(3):918--932, 2018.

\bibitem{Gudi10}
T.~Gudi.
\newblock A new error analysis for discontinuous finite element methods for
  linear elliptic problems.
\newblock {\em Math. Comp.}, 79(272):2169--2189, 2010.

\bibitem{LJS17}
N.~L\"uthen, M.~Juntunen, and R.~Stenberg.
\newblock An improved a priori error analysis of nitsche’s method for robin
  boundary conditions.
\newblock {\em Numerische Mathematik}, 2017.
\newblock https://doi.org/10.1007/s00211-017-0927-1.

\bibitem{Sav97}
G.~Savar\'e.
\newblock Regularity and perturbation results for mixed second order elliptic
  problems.
\newblock {\em Comm. Partial Differential Equations}, 22(5-6):869--899, 1997.

\end{thebibliography}

\end{document}